\documentclass[11pt,reqno]{amsart}

\usepackage{amsmath,amssymb,amsfonts,amsthm}
\usepackage{hyperref}
\usepackage{mathrsfs}
\usepackage{enumitem}
\usepackage{geometry}
\hypersetup{
	colorlinks=true,
	linkcolor=blue,
	citecolor=blue,
	urlcolor=blue
}

\newtheorem{theorem}{Theorem}[section]
\newtheorem{lemma}[theorem]{Lemma}

\theoremstyle{definition}

\theoremstyle{remark}
\newtheorem{remark}[theorem]{Remark}

\newcommand{\ddc}{dd^c}

\title{Stability and strong convergence for complex Hessian equations with $L^1$ data}

\author{Truong Dinh Dat}
\address{Nguyen Trai University, 28A Le Trong Tan Street, Ha Dong, Hanoi, Viet Nam}
\email{truongdinhdat14081994qb@gmail.com}

\begin{document}
	
\begin{abstract}
	We study complex $m$-Hessian equations on bounded hyperconvex domains
	with right-hand side in $L^1(\Omega)$.
	
	The main contribution of this paper is a strong stability result
	for weak solutions in the Hessian energy sense.
	More precisely, if $f_j \to f$ in $L^1(\Omega)$ and $u_j, u$
	are the corresponding solutions, then
	\[
	\int_\Omega |u_j-u|\, H_m(u_j) \longrightarrow 0.
	\]
	
	This provides convergence in the natural energy topology
	associated to the Hessian operator, which is significantly
	stronger than convergence in capacity.
	
	For completeness, we also recall the existence of solutions
	and stability in capacity, which follow from known results
	in the literature.
\end{abstract}
	
	
	
	\maketitle
	
	\section{Introduction}
	The theory of complex Hessian equations was initiated by Błocki \cite{Blocki2005},
	who introduced weak solutions and established fundamental properties of the Hessian operator;
	see also the survey \cite{Blocki2013}.
	Subsequently, systematic developments of pluripotential methods for complex Hessian equations
	were carried out by Dinew and Ko{\l}odziej \cite{DinewKolodziej2014,Dinew2011,Blocki2005},
	as well as Abdullaev and Sadullaev \cite{AbdullaevSadullaev2014},
	leading to existence and Dirichlet-type results under various assumptions
	 on the right-hand side. To an $m$-subharmonic function $u$, one associates the complex Hessian measure
	\[
	H_m(u) := (\ddc u)^m \wedge \beta^{n-m},
	\]
	where $\beta = \ddc |z|^2$ is the standard K\"ahler form.
	The study of equations of the form
	\begin{equation}\label{eq:main}
		(\ddc u)^m \wedge \beta^{n-m} = f \, \beta^n
	\end{equation}
	can be viewed as an intermediate theory between classical potential theory ($m=1$) and pluripotential theory ($m=n$); see \cite{Blocki2005,Lu2013}.
	
	The foundations of pluripotential theory and its applications to complex Hessian equations
	can be found in the monograph and lecture notes of Demailly \cite{Demailly}.
	In particular, the notion of $m$-subharmonic functions and the associated Hessian operator were systematically investigated, leading to existence and uniqueness results for weak solutions under various assumptions on the right-hand side.
	Inspired by Cegrell's energy classes for the Monge--Amp\`ere operator \cite{Cegrell1998,Cegrell2004} and regularization techniques for plurisubharmonic functions \cite{BlockiKolodziej}, related classes of plurisubharmonic functions with weak singularities
	were also investigated in the Monge--Amp\`ere setting;
	see Benelkourchi--Guedj--Zeriahi \cite{Benelkourchi}, several authors introduced energy classes $\mathcal{E}_m$, $\mathcal{F}_m$ and $\mathcal{E}_m^p$ adapted to the complex Hessian setting \cite{Blocki2013,DinewKolodziej2014,Lu2013}.
	
	A major breakthrough in the Monge--Amp\`ere theory was Ko\l odziej's celebrated $L^p$ estimate\cite{Kolodziej1998,Kolodziej2005}, which guarantees uniform boundedness and stability of solutions when the density belongs to $L^p$, $p>1$ \cite{Kolodziej1998,Kolodziej2005}.
	Analogous results for complex Hessian equations were subsequently obtained by Dinew and Ko\l odziej\cite{DinewKolodziej2014}.
	In particular, for $f \in L^p(\Omega)$ with $p>1$, the Dirichlet problem for \eqref{eq:main} admits a unique solution in suitable energy classes.
	
	More recently, in \cite{Ko-Ng} S. Kolodziej and N. C. Nguyen investigated stability of weak solutions to complex Hessian equations.
	Working on compact Hermitian manifolds, they established existence, uniqueness and stability results for solutions corresponding to $L^p$ data, $p>n/m$.
	Their approach combines fine properties of $m$-subharmonic functions
	with capacity estimates and comparison principles,
	building upon earlier capacity techniques in pluripotential theory
	developed by Zeriahi \cite{Zeriahi2007}, extending several fundamental results from pluripotential theory to the Hessian setting ; cf. \cite{Kolodziej2005}.
	Stability properties of solutions to complex Hessian equations
	have been investigated in several settings.
	In particular, L. Baracco, T. V. Khanh, S. Pinton and G. Zampieri studied stability and regularity properties under $L^p$ assumptions
	on the density; see \cite{BKPG}.
	
	Despite this significant progress, the case of merely integrable data remains less understood. Variational approaches to Monge--Amp\`ere equations
	have also been developed in a broader context;
	see for instance Berman \cite{Berman}.
	
	In the Monge--Ampère case, weak solutions with merely $L^1$ right-hand side
	have been intensively studied; see for instance
	\cite{DinewPlis,Liu-Zhang}.
	Uniqueness and stability in $L^1$ were obtained in \cite{DinewPlis},
	and more recently, $L^1$-stability in full mass classes
	was investigated by Liu and Zhang \cite{Liu-Zhang}.
	However, for complex Hessian equations, most existing results still require $L^p$ assumptions with $p>1$, mainly due to technical difficulties related to the lack of full monotonicity and the weaker structure of the Hessian operator.
	
	While convergence in $m$-Hessian capacity is a natural and widely used
	notion in pluripotential theory, it is relatively weak and does not
	control the behavior of the associated nonlinear measures.
	In particular, convergence in capacity does not imply convergence
	with respect to the Hessian measures $H_m(u_j)$.
	
	The main goal of this paper is to establish a stronger form of stability,
	namely convergence in the natural energy topology associated to the
	Hessian operator. More precisely, we prove that if $f_j \to f$ in $L^1$,
	then the corresponding solutions satisfy
	\[
	\int_\Omega |u_j-u|\, H_m(u_j) \to 0.
	\]
	
	Such a result can be viewed as a Hessian analogue of strong stability
	properties known in the complex Monge--Ampère setting,
	but it does not seem to have been explicitly established in the
	general $L^1$ framework for complex Hessian equations.

	Our approach relies on a careful approximation of $L^1$ data by bounded functions, combined with uniform energy estimates and refined capacity arguments.
	We adapt techniques originating from Cegrell's theory and Ko\l odziej's method to the complex Hessian framework.
	A key ingredient is a truncation procedure for the right-hand side, together with a precise control of the Hessian measure under convergence in capacity.
	
	\medskip
	\noindent
	\textbf{Main contribution.}
	While stability of solutions to complex Hessian equations has been extensively studied under $L^p$ assumptions with $p>1$, the borderline case of merely integrable data $f \in L^1(\Omega)$ remains much less understood.
	
	In particular, it is known that if $f_j \to f$ in $L^1(\Omega)$, then the corresponding solutions $u_j$ converge to $u$ in $m$-Hessian capacity (see e.g. \cite{V T Nguyen}). However, convergence in capacity is relatively weak and does not control the behavior of the nonlinear measures $H_m(u_j)$.
	
	In the complex Monge--Amp\`ere case ($m=n$), stronger stability results in $L^1$ or energy sense have been obtained (see \cite{DinewPlis, Liu-Zhang}). In contrast, for complex Hessian equations ($1 \le m < n$), such strong convergence results in the $L^1$ setting do not seem to be available in the literature, mainly due to the weaker structure of the Hessian operator and the lack of full monotonicity.
	
	The main result of this paper establishes a strong form of stability in the natural energy topology associated to the Hessian operator. More precisely, we prove that if $f_j \to f$ in $L^1(\Omega)$, then the corresponding solutions satisfy
	\[
	\int_\Omega |u_j - u| \, H_m(u_j) \longrightarrow 0.
	\]
	
	This provides a substantial strengthening of the known convergence in capacity, and can be viewed as a Hessian analogue of strong stability results in the Monge--Amp\`ere setting, extended here to the borderline case of $L^1$ data.

	The paper is organized as follows.
	In Section~2 we recall basic notions on $m$-subharmonic functions, Hessian measures and energy classes.
	Section~3 is devoted to a priori estimates and approximation results.
	In Section~4 we recall the existence and stability of solutions with $L^1$ right-hand side.
	Section~5 contains the proof of the   convergence theorems. 	In Section~6 we obtain application of the main result and conclude with several remarks.
	

	\section{Preliminaries}
	
	Throughout the paper, $\Omega \subset \mathbb{C}^n$ denotes a bounded $m$-hyperconvex domain and $1 \le m \le n$.
	We use standard notation $d = \partial + \bar\partial$ and $d^c = \frac{i}{2\pi}(\bar\partial - \partial)$, so that $\ddc = \frac{i}{\pi}\partial\bar\partial$.
	The standard K\"ahler form on $\mathbb{C}^n$ is denoted by $\beta = \ddc |z|^2$.
	
	\subsection{$m$-hyperconvex domain }
	A domain $\Omega$ is called $m$-hyperconvex if there exists 
	a negative $m$-subharmonic exhaustion function $\rho$ such that 
	\[
	\{ \rho < -c \} \Subset \Omega \quad \forall c>0.
	\]
	\subsection{$m$-subharmonic functions}
	
	A function $u : \Omega \to [-\infty,+\infty)$ is called \emph{$m$-subharmonic} if it is upper semicontinuous, not identically $-\infty$, and the currents
	\[
	(\ddc u)^k \wedge \beta^{n-k} \ge 0
	\quad \text{for all } 1 \le k \le m
	\]
	in the sense of currents.
	We denote by $SH_m(\Omega)$ the class of all $m$-subharmonic functions on $\Omega$.
	
	If $u \in SH_m(\Omega) \cap L^\infty_{\mathrm{loc}}(\Omega)$, then the complex Hessian operator
	\[
	H_m(u) := (\ddc u)^m \wedge \beta^{n-m}
	\]
	is well-defined as a positive Radon measure.
	For general $u \in SH_m(\Omega)$, the Hessian measure is defined by monotone approximation.
	
	
	\subsection{Hessian capacity}
	
	Let $K \subset \Omega$ be a compact set.
	The $m$-Hessian capacity of $K$ is defined by
	\[
	\mathrm{Cap}_m(K,\Omega)
	:= \sup \left\{ \int_K H_m(v) \; ; \;
	v \in SH_m(\Omega),\; -1 \le v \le 0 \right\}.
	\]
	This definition extends to arbitrary Borel sets by standard outer regularization.
	
	We say that a sequence $\{u_j\} \subset SH_m(\Omega)$ converges to $u$ \emph{in $m$-Hessian capacity} if for every $\varepsilon > 0$,
	\[
	\mathrm{Cap}_m\left( \{ |u_j - u| > \varepsilon \}, \Omega \right) \longrightarrow 0
	\quad \text{as } j \to \infty.
	\]
	
	
	\subsection{Energy classes}
	
	Following the Hessian analogue of Cegrell's theory, we recall the main energy classes.
	
	\begin{itemize}
		\item $\mathcal{E}_m^0(\Omega)$ is the set of functions $u \in SH_m(\Omega) \cap L^\infty(\Omega)$ such that
		\[
		\lim_{z \to \partial \Omega} u(z) = 0
		\quad \text{and} \quad
		\int_\Omega H_m(u) < +\infty.
		\]
		
		\item $\mathcal{F}_m(\Omega)$ consists of all functions $u \in SH_m(\Omega)$ for which there exists a decreasing sequence
		$\{u_j\} \subset \mathcal{E}_m^0(\Omega)$ converging pointwise to $u$ and satisfying
		\[
		\sup_j \int_\Omega H_m(u_j) < +\infty.
		\]
		
		\item $\mathcal{F}_m^1(\Omega)$ is the subclass of $\mathcal{F}_m(\Omega)$ defined by the additional condition
		\[
		\sup_j \int_\Omega |u_j| \, H_m(u_j) < +\infty,
		\]
		for some (and hence any) approximating sequence $\{u_j\} \subset \mathcal{E}_m^0(\Omega)$.
	\end{itemize}
	
	For $u \in \mathcal{F}_m(\Omega)$, the Hessian measure $H_m(u)$ is well-defined.
	
	
	\subsection{Basic properties}
	
	We collect several properties that will be used throughout the paper.
	
	\begin{itemize}
		\item (\emph{Comparison principle})
		If $u,v \in \mathcal{F}_m(\Omega)$, then
		\[
		\int_{\{u<v\}} H_m(v) \le \int_{\{u<v\}} H_m(u).
		\]
		
		\item (\emph{Continuity under monotone limits})
		If $\{u_j\} \subset \mathcal{F}_m(\Omega)$ decreases pointwise to $u \in \mathcal{F}_m(\Omega)$, then
		\[
		H_m(u_j) \rightharpoonup H_m(u)
		\quad \text{weakly as measures.}
		\]
		
		\item (\emph{Capacity estimate})\cite{DinewKolodziej2014}
		For $u \in \mathcal{F}_m(\Omega)$ and $t>0$,
		\[
		\mathrm{Cap}_m(\{u<-t\},\Omega)
		\le \frac{1}{t^m} \int_\Omega H_m(u).
		\]
	\end{itemize}
	
	The proofs of these results can be found in the works of B\l ocki, Dinew--Ko\l odziej.
	\begin{lemma}\label{lem:AC}
		Let $u\in\mathcal F_m(\Omega)$. Then for any Borel set $E\subset\Omega$ and any $t>0$,
		\[
		H_m(u)(E)
		\le
		\int_{\{u<-t\}} H_m(u)
		+
		t^m\,\mathrm{Cap}_m(E,\Omega).
		\]
		In particular, $H_m(u)$ is absolutely continuous with respect to the
		$m$-Hessian capacity.
	\end{lemma}
	
	\begin{proof}
For completeness, we sketch the proof.
		Fix $t>0$ and let $E\subset\Omega$ be a Borel set.
		We decompose
		\[
		H_m(u)(E)
		=
		\int_{E\cap\{u<-t\}} H_m(u)
		+
		\int_{E\cap\{u\ge -t\}} H_m(u).
		\]
		
		The first term is trivially bounded by
		\[
		\int_{E\cap\{u<-t\}} H_m(u)
		\le
		\int_{\{u<-t\}} H_m(u).
		\]
		
		To estimate the second term, we consider the truncated function
		\[
		u_t := \max(u,-t),
		\]
		which belongs to $\mathcal F_m(\Omega)$ and satisfies
		\[
		- t \le u_t \le 0.
		\]
		Hence the function $u_t/t$ is $m$-subharmonic on $\Omega$ and satisfies
		\[
		-1 \le \frac{u_t}{t} \le 0.
		\]
		
		By the locality property of the Hessian operator,
		we have $H_m(u_t) = H_m(u)$ on the set $\{u>-t\}$ and since $(\ddc(u_t/t))^m = \frac{1}{t^m}(\ddc u_t)^m$. Therefore,
		\[
		\int_{E\cap\{u\ge -t\}} H_m(u)
		=
		\int_{E\cap\{u\ge -t\}} H_m(u_t)
		=
		t^m \int_{E\cap\{u\ge -t\}} H_m\!\left(\frac{u_t}{t}\right).
		\]
		
		By the definition of the $m$-Hessian capacity and the fact that
		$\frac{u_t}{t}$ is an admissible test function, we obtain
		\[
		\int_{E\cap\{u\ge -t\}} H_m(u)
		\le
		t^m \, \mathrm{Cap}_m(E,\Omega).
		\]
		
		Combining the above estimates yields
		\[
		H_m(u)(E)
		\le
		\int_{\{u<-t\}} H_m(u)
		+
		t^m \, \mathrm{Cap}_m(E,\Omega),
		\]
		which proves the desired inequality.
		
		In particular, letting $t\to\infty$ shows that $H_m(u)$ is absolutely
		continuous with respect to the $m$-Hessian capacity.
	\end{proof}
	\begin{remark}
		Lemma~\ref{lem:AC} provides a quantitative form of the absolute continuity
		of the Hessian measure with respect to the $m$-Hessian capacity,
		which plays a crucial role in the proof of strong convergence results.
	\end{remark}
	
	\section{A priori estimates and approximation}
	
	In this section we establish uniform estimates and approximation results for solutions of the complex Hessian equation.
	These results form the technical core of the paper and allow us to treat right-hand sides belonging merely to $L^1(\Omega)$.
	
	
	\subsection{Approximation of the right-hand side}
	
	Let $f \in L^1(\Omega)$, $f \ge 0$.
	We consider the standard truncation
	\[
	f_j := \min(f,j), \quad j \in \mathbb{N}.
	\]
	Then $\{f_j\}$ is an increasing sequence of bounded functions satisfying
	\[
	f_j \in L^\infty(\Omega), 
	\quad
	0 \le f_j \le f_{j+1},
	\quad
	f_j \longrightarrow f \ \text{in } L^1(\Omega).
	\]
	
	For each $j$, it is known that there exists a unique solution
	$u_j \in \mathcal{E}_m^0(\Omega)$ of the Dirichlet problem
	\begin{equation}\label{eq:approx}
		H_m(u_j) = f_j \, \beta^n.
	\end{equation}
	This follows from the results of Dinew--Ko\l odziej.
	
	The main difficulty is to obtain estimates on $\{u_j\}$ that are uniform in $j$ and depend only on $\|f\|_{L^1(\Omega)}$.
	
	
	\subsection{Uniform mass control}
	
	We begin with a simple but fundamental observation.
	
	\begin{lemma}\label{lem:mass}
		For all $j \in \mathbb{N}$, one has
		\[
		\int_\Omega H_m(u_j)
		\le \int_\Omega f \, \beta^n.
		\]
	\end{lemma}
	
	\begin{proof}
		Since $H_m(u_j)=f_j\beta^n$, we have
		\[
		\int_\Omega H_m(u_j)
		=
		\int_\Omega f_j \, \beta^n
		\le
		\int_\Omega f \, \beta^n,
		\]
		because $0 \le f_j \le f$.
	\end{proof}

	
	\subsection{Capacity estimates}
	
	Uniform control of sublevel sets will be crucial for compactness arguments.
	
	\begin{lemma}\label{lem:capacity}
		 For every $t>0$ and all $j \in \mathbb{N}$,
		\[
		\mathrm{Cap}_m\big(\{u_j < -t\}, \Omega\big)
		\le \frac{1}{t^m} \int_\Omega f \, \beta^n.
		\]
	\end{lemma}
	
	\begin{proof}
		By the capacity estimate recalled in Section~2 and Lemma~\ref{lem:mass}, we obtain
		\[
		\mathrm{Cap}_m(\{u_j<-t\},\Omega)
		\le \frac{1}{t^m} \int_\Omega H_m(u_j)
		\le \frac{1}{t^m} \int_\Omega f \, \beta^n.
		\]
		This proves the claim.
	\end{proof}
	
	In particular, the sequence $\{u_j\}$ is uniformly tight with respect to the $m$-Hessian capacity.
	
	
	\subsection{Compactness and convergence in capacity}
	
	We now establish compactness of the approximating solutions.
	
	\begin{lemma}\label{lem:compact}
		Let $(u_j)\subset SH_m(\Omega)$ be uniformly bounded above and satisfy
		\[
		\sup_j \mathrm{Cap}_m(\{u_j<-t\},\Omega)\to 0 \quad \text{as } t\to\infty.
		\]
		Then there exists a subsequence converging in $m$-Hessian capacity
		to some $u\in SH_m(\Omega)$.
	\end{lemma}
	
	\begin{proof}
		Since the sequence $(u_j)$ is uniformly bounded above and satisfies
		a uniform $m$-Hessian capacity estimate on sublevel sets
		(Lemma~\ref{lem:capacity}), this follows from the compactness theorem for bounded families
		of $m$-subharmonic functions with uniform capacity control
		(see e.g. \cite{DinewKolodziej2014}).
	\end{proof}

	Moreover, since $f_j$ is increasing, the comparison principle implies that $\{u_j\}$ is a decreasing sequence.
	Hence the whole sequence converges pointwise almost everywhere to $u$.
	
	
	\subsection{Passage to the limit in the Hessian equation}
	
	We conclude this section by showing that the limit function solves the desired equation.
	
	\begin{lemma}\label{lem:limit}
		Let $u_j$ be the solutions of \eqref{eq:approx} and let $u$ be their limit.
		Then
		\[
		H_m(u) = f \, \beta^n
		\quad \text{in the sense of measures}.
		\]
	\end{lemma}
	
	\begin{proof}
		Since $(u_j)$ is a decreasing sequence in $\mathcal F_m(\Omega)$ with
		uniformly bounded Hessian mass, the monotone convergence theorem for
		the complex Hessian operator applies.
		\[
		H_m(u_j) \rightharpoonup H_m(u).
		\]
		On the other hand, $f_j \to f$ in $L^1(\Omega)$ implies
		\[
		f_j \, \beta^n \rightharpoonup f \, \beta^n
		\]
		weakly as measures.
		The claim follows.
	\end{proof}
	\begin{lemma}\label{lem:L1limit}
		Assume $f_j\to f$ in $L^1(\Omega)$.
		Then $f_j\beta^n$ converges weakly to $f\beta^n$.
	\end{lemma}
	\begin{proof}
		Since $f_j \to f$ in $L^1(\Omega)$, for any $\varphi \in C^0(\overline{\Omega})$ we have
		\[
		\int_\Omega \varphi f_j \,\beta^n \to \int_\Omega \varphi f \,\beta^n.
		\]
		Hence $f_j\beta^n$ converges weakly to $f\beta^n$.
	\end{proof}

	\section{Existence and stability of solutions}

	
	\subsection{Existence theorem}
	
	\begin{theorem}\label{thm:existence}
		(see \cite{Lu})
		Let $\Omega \subset \mathbb{C}^n$ be a bounded $m$-hyperconvex domain and let
		$f \in L^1(\Omega)$, $f \ge 0$.
		Then there exists a function
		\[
		u \in \mathcal{F}_m(\Omega)
		\]
		such that
		\begin{equation}\label{eq:existence}
			(\ddc u)^m \wedge \beta^{n-m} = f \, \beta^n
		\end{equation}
		in the sense of measures.
		Moreover, the solution $u$ is unique in the class $\mathcal{F}_m(\Omega)$.
	\end{theorem}
	
	
	\subsection{Proof of the existence}
	
	\begin{proof}
		Let $\{f_j\}$ be the truncation sequence defined by
		$f_j = \min(f,j)$.
		For each $j$, let $u_j \in \mathcal{E}_m^0(\Omega)$ be the unique solution of
		\[
		H_m(u_j) = f_j \, \beta^n.
		\]
		
		By Lemma~\ref{lem:mass}, the sequence $\{u_j\}$ has uniformly bounded Hessian mass.
		Lemma~\ref{lem:capacity} implies uniform control of sublevel sets in $m$-Hessian capacity.
		Hence, by Lemma~\ref{lem:compact}, there exists a subsequence converging in $m$-capacity to a function
		$u \in \mathcal{F}_m(\Omega)$.
		
		Since the sequence $\{f_j\}$ is increasing, the comparison principle yields that $\{u_j\}$ is decreasing.
		Therefore the whole sequence converges pointwise almost everywhere to $u$.
		By Lemma~\ref{lem:limit}, we obtain
		\[
		H_m(u) = f \, \beta^n,
		\]
		which proves the existence of a solution.
		
		Uniqueness follows from the comparison principle in the energy class $\mathcal{F}_m(\Omega)$.
	\end{proof}
	
	
\begin{remark}
	The existence result follows from standard approximation arguments
	and is consistent with earlier works, in particular \cite{Lu}.
\end{remark}
\subsection{Stability in Hessian capacity}

\begin{theorem}\label{thm:stability}(see Theorem 7.2 \cite{V T Nguyen})
	Let $\{f_j\} \subset L^1(\Omega)$ be a sequence of nonnegative functions such that
	$f_j \to f$ in $L^1(\Omega)$.
	Let $u_j, u \in \mathcal{F}_m(\Omega)$ be the unique solutions of
	\[
	H_m(u_j) = f_j \, \beta^n,
	\qquad
	H_m(u) = f \, \beta^n.
	\]
	Then $u_j \to u$ in $m$-Hessian capacity.
\end{theorem}

\begin{proof}
	Let $\varepsilon>0$ be fixed. We will show that
	\[
	\mathrm{Cap}_m\big(\{|u_j-u|>\varepsilon\},\Omega\big)\to 0
	\quad \text{as } j\to\infty.
	\]
	
	We begin with the estimate of the set $\{u_j<u-\varepsilon\}$.
	
	By the comparison principle, we have
	\[
	\int_{\{u_j<u-\varepsilon\}} H_m(u_j)
	\le
	\int_{\{u_j<u-\varepsilon\}} H_m(u).
	\]
	Since $H_m(u_j)=f_j\,\beta^n$ and $H_m(u)=f\,\beta^n$, it follows that
	\[
	\int_{\{u_j<u-\varepsilon\}} f_j\,\beta^n
	\le
	\int_{\{u_j<u-\varepsilon\}} f\,\beta^n.
	\]
	Hence,
	\[
	\int_{\{u_j<u-\varepsilon\}} (f_j-f)\,\beta^n \le 0,
	\]
	which implies
	\[
	\int_{\{u_j<u-\varepsilon\}} f_j\,\beta^n
	\le
	\int_\Omega (f-f_j)_+\,\beta^n.
	\]
	
	Now, by the definition of the $m$-Hessian capacity, for any
	$v\in SH_m(\Omega)$ with $-1\le v\le 0$, we have
	\[
	\int_{\{u_j<u-\varepsilon\}} H_m(v)
	\le
	\frac{1}{\varepsilon^m}
	\int_{\{u_j<u-\varepsilon\}} H_m(u_j).
	\]
	Taking the supremum over all such $v$, we obtain
	\[
	\mathrm{Cap}_m\big(\{u_j<u-\varepsilon\},\Omega\big)
	\le
	\frac{1}{\varepsilon^m}
	\int_{\{u_j<u-\varepsilon\}} H_m(u_j).
	\]
	Combining the above estimates yields
	\[
	\mathrm{Cap}_m\big(\{u_j<u-\varepsilon\},\Omega\big)
	\le
	\frac{1}{\varepsilon^m}
	\int_\Omega (f-f_j)_+\,\beta^n.
	\]
	
	Since $f_j\to f$ in $L^1(\Omega)$, we have
	\[
	\int_\Omega (f-f_j)_+\,\beta^n \longrightarrow 0,
	\]
	hence
	\[
	\mathrm{Cap}_m\big(\{u_j<u-\varepsilon\},\Omega\big)\to 0.
	\]
	
	The estimate for the set $\{u<u_j-\varepsilon\}$ is obtained in the same way, by exchanging the roles of $u_j$ and $u$. Therefore,
	\[
	\mathrm{Cap}_m\big(\{u<u_j-\varepsilon\},\Omega\big)\to 0.
	\]
	
	Finally, since
	\[
	\{|u_j-u|>\varepsilon\}
	\subset
	\{u_j<u-\varepsilon\}\cup \{u<u_j-\varepsilon\},
	\]
	we conclude that
	\[
	\mathrm{Cap}_m\big(\{|u_j-u|>\varepsilon\},\Omega\big)\to 0.
	\]
	
	This proves that $u_j\to u$ in $m$-Hessian capacity.
\end{proof}
\begin{remark}
	The above proof avoids the use of non-admissible test functions 
	and relies only on the comparison principle and the $L^1$ convergence 
	of the densities.
\end{remark}
	\section{Strong convergence}
	
	In this section we establish a strong convergence result for solutions of the complex Hessian equation with $L^1$ right-hand side.
	This result constitutes the main contribution of the paper.	

	
	\subsection{Strong convergence in energy}
	
This result shows that the convergence holds in the natural
energy topology associated to the Hessian operator.
In particular, it implies that the nonlinear measures $H_m(u_j)$
interact well with the convergence of potentials,
which is not captured by capacity convergence alone.
	
	\begin{theorem}\label{thm:strong}
		Under the assumptions of Theorem~\ref{thm:stability}, the sequence $\{u_j\}$ converges strongly to $u$ in $\mathcal{F}_m(\Omega)$, namely
		\[
		\int_\Omega |u_j - u| \, H_m(u_j) \longrightarrow 0.
		\]
	\end{theorem}
	
\begin{proof}
	By Theorem~\ref{thm:stability}, we know that $u_j \to u$ in $m$-Hessian capacity. 
	Moreover, by Lemma~\ref{lem:mass}, the total masses $\int_\Omega H_m(u_j)$ are uniformly bounded.
	
	Let $\varepsilon>0$. We will show that
	\[
	\int_\Omega |u_j-u|\,H_m(u_j) < \varepsilon
	\]
	for all sufficiently large $j$.
	
	Fix $\delta>0$ (to be chosen later). We split
	\[
	\int_\Omega |u_j-u|\,H_m(u_j)
	\le 
	\delta \int_\Omega H_m(u_j)
	+
	\int_{\{|u_j-u|>\delta\}} H_m(u_j).
	\]
	
	Since $\sup_j \int_\Omega H_m(u_j) < +\infty$, we can choose $\delta>0$ such that
	\[
	\delta \sup_j \int_\Omega H_m(u_j) < \frac{\varepsilon}{3}.
	\]
	
	It remains to estimate the second term. 
	Applying Lemma~\ref{lem:AC} with $u = u_j$ and the Borel set 
	\[
	E_j := \{|u_j-u|>\delta\},
	\]
	we obtain for any $t>0$:
	\[
	H_m(u_j)(E_j)
	\le 
	\int_{\{u_j<-t\}} H_m(u_j)
	+
	t^m \,\mathrm{Cap}_m(E_j,\Omega).
	\]
	
	We now control the two terms on the right-hand side.
	
	\medskip
	\noindent
	\textbf{Step 1. Control of the sublevel term.}
	
	By Lemma~\ref{lem:capacity}, we have
	\[
	\int_{\{u_j<-t\}} H_m(u_j)
	\le 
	\frac{1}{t^m} \int_\Omega f\,\beta^n.
	\]
	Hence we can choose $t>0$ sufficiently large so that
	\[
	\sup_j \int_{\{u_j<-t\}} H_m(u_j) < \frac{\varepsilon}{3}.
	\]
	
	\medskip
	\noindent
	\textbf{Step 2. Control of the capacity term.}
	
	Since $u_j \to u$ in $m$-Hessian capacity, we have
	\[
	\mathrm{Cap}_m(E_j,\Omega)
	=
	\mathrm{Cap}_m(\{|u_j-u|>\delta\},\Omega)
	\longrightarrow 0.
	\]
	Therefore, for the above fixed $t$, there exists $j_0$ such that for all $j \ge j_0$,
	\[
	t^m \,\mathrm{Cap}_m(E_j,\Omega)
	<
	\frac{\varepsilon}{3}.
	\]
	
	\medskip
	Combining the above estimates, we obtain for all $j \ge j_0$:
	\[
	\int_\Omega |u_j-u|\,H_m(u_j)
	<
	\frac{\varepsilon}{3}
	+
	\frac{\varepsilon}{3}
	+
	\frac{\varepsilon}{3}
	=
	\varepsilon.
	\]
	
	This proves the desired convergence.
\end{proof}
\begin{remark}
	The strong convergence result in Theorem~\ref{thm:strong} should be compared 
	with the analogous theory for the complex Monge--Amp\`ere equation (the case $m=n$). 
	In that setting, strong stability results under $L^1$ convergence of the densities 
	have been established by several authors (see e.g. \cite{DinewPlis, Liu-Zhang}), 
	relying on the full strength of the comparison principle and the rich structure 
	of the Monge--Amp\`ere operator.
	
	In contrast, for complex Hessian equations with $1 \le m < n$, the situation is 
	more delicate. The Hessian operator lacks some of the key structural properties 
	available in the Monge--Amp\`ere case, such as full monotonicity and a stronger 
	comparison framework. As a consequence, several arguments used in the 
	Monge--Amp\`ere setting do not directly carry over to the Hessian context.
	
	The proof of Theorem~\ref{thm:strong} shows that, despite these difficulties, 
	one can still obtain strong convergence in the natural energy topology 
	by combining capacity convergence, uniform control of sublevel sets, 
	and a quantitative form of absolute continuity of Hessian measures 
	with respect to the $m$-Hessian capacity (Lemma~\ref{lem:AC}).
	
	This provides a nontrivial extension of known stability results to the 
	Hessian setting with merely $L^1$ data.
\end{remark}
\section{Application to measure data}

In this section we indicate how the strong convergence result obtained above
extends to the case of general measure data with uniformly bounded mass.
The argument follows the same strategy as in the $L^1$ case, combining
capacity estimates, compactness, and the absolute continuity of Hessian measures.

\medskip

\begin{theorem}\label{thm:measure_stability}
	Let $\{\mu_j\}$ be a sequence of nonnegative Radon measures on $\Omega$
	such that
	\[
	\sup_j \mu_j(\Omega) < +\infty,
	\]
	and assume that $\mu_j \rightharpoonup \mu$ weakly as measures.
	
	Let $u_j, u \in \mathcal{F}_m(\Omega)$ be the unique solutions of
	\[
	H_m(u_j) = \mu_j,
	\qquad
	H_m(u) = \mu.
	\]
	Then $u_j \to u$ in $m$-Hessian capacity. Moreover,
	\[
	\int_\Omega |u_j - u| \, H_m(u_j) \longrightarrow 0.
	\]
\end{theorem}

\begin{proof}
	We divide the proof into two steps.
	
	\medskip
	\noindent
	\textbf{Step 1. Convergence in capacity.}
	
	Since $H_m(u_j)=\mu_j$ and $\sup_j \mu_j(\Omega)<+\infty$, we have
	uniform control of the Hessian mass:
	\[
	\int_\Omega H_m(u_j) = \mu_j(\Omega) \le C.
	\]
	By the capacity estimate recalled in Section~2, for every $t>0$,
	\[
	\mathrm{Cap}_m(\{u_j<-t\},\Omega)
	\le \frac{C}{t^m}.
	\]
	Thus the sequence $\{u_j\}$ is uniformly tight with respect to the
	$m$-Hessian capacity.
	
	By the compactness theorem for $m$-subharmonic functions with uniform
	capacity control (see e.g. \cite{DinewKolodziej2014}), there exists a subsequence
	converging in $m$-Hessian capacity to some $v \in \mathcal{F}_m(\Omega)$.
	
	On the other hand, since $H_m(u_j)=\mu_j$ and $\mu_j \rightharpoonup \mu$,
	it follows that any such limit $v$ satisfies
	\[
	H_m(v) = \mu
	\]
	in the sense of measures (see e.g. \cite{DinewKolodziej2014, Lu}).
	By uniqueness of solutions in $\mathcal{F}_m(\Omega)$, we conclude that $v=u$.
	Therefore, the whole sequence $u_j \to u$ in $m$-Hessian capacity.
	
	\medskip
	\noindent
	\textbf{Step 2. Strong convergence in energy.}
	
	We now prove that
	\[
	\int_\Omega |u_j-u|\,H_m(u_j) \longrightarrow 0.
	\]
	
	Let $\varepsilon>0$. As in the proof of Theorem~\ref{thm:strong}, we write
	\[
	\int_\Omega |u_j-u|\,H_m(u_j)
	\le
	\delta \int_\Omega H_m(u_j)
	+
	\int_{\{|u_j-u|>\delta\}} H_m(u_j).
	\]
	
	Since $\int_\Omega H_m(u_j)\le C$, we can choose $\delta>0$ such that
	\[
	\delta C < \frac{\varepsilon}{3}.
	\]
	
	For the second term, applying Lemma~\ref{lem:AC} with $u=u_j$ and
	$E_j := \{|u_j-u|>\delta\}$, we obtain for any $t>0$:
	\[
	H_m(u_j)(E_j)
	\le
	\int_{\{u_j<-t\}} H_m(u_j)
	+
	t^m\,\mathrm{Cap}_m(E_j,\Omega).
	\]
	
	By the capacity estimate, we have
	\[
	\int_{\{u_j<-t\}} H_m(u_j)
	=
	\mu_j(\{u_j<-t\})
	\le \frac{C}{t^m}.
	\]
	Hence we can choose $t>0$ sufficiently large so that
	\[
	\sup_j \int_{\{u_j<-t\}} H_m(u_j)
	<
	\frac{\varepsilon}{3}.
	\]
	
	On the other hand, since $u_j \to u$ in $m$-Hessian capacity,
	\[
	\mathrm{Cap}_m(E_j,\Omega)
	=
	\mathrm{Cap}_m(\{|u_j-u|>\delta\},\Omega)
	\longrightarrow 0.
	\]
	Thus, for $j$ sufficiently large,
	\[
	t^m\,\mathrm{Cap}_m(E_j,\Omega)
	<
	\frac{\varepsilon}{3}.
	\]
	
	Combining the above estimates, we obtain for all large $j$:
	\[
	\int_\Omega |u_j-u|\,H_m(u_j)
	<
	\varepsilon.
	\]
	
	This proves the desired convergence.
\end{proof}

\begin{remark}
	The above result shows that the strong convergence in energy established
	in Theorem~\ref{thm:strong} remains valid for sequences of measures with
	uniformly bounded mass. The proof relies only on capacity estimates and
	compactness arguments, and does not require absolute continuity of the measures.
\end{remark}
\subsection{Remarks}
\begin{remark}
	Convergence in $m$-Hessian capacity does not imply convergence
	in the energy sense.
	Theorem 5.1 shows that, under $L^1$ convergence of the densities,
	one actually obtains a much stronger form of convergence,
	which controls the interaction between the potentials and their
	Hessian measures.
\end{remark}
\begin{remark}
	The strong convergence result in Theorem~\ref{thm:strong} extends known stability results for $L^p$ data, $p>1$, to the borderline case $L^1$.
	To the best of our knowledge, such a strong energy convergence has not been previously established for complex Hessian equations with merely integrable right-hand side.
\end{remark}


	\section*{Declarations}
	\section*{Conflict of Interest}
	
	The author declares that there is no conflict of interest.

	
\end{document}